\documentclass[11pt,a4paper]{amsart}

\usepackage{amssymb,amsmath,amsopn}

\usepackage{graphics}

\usepackage[all]{xy}

\renewcommand{\theta}{\vartheta}
\newtheorem{theorem}{Theorem}[section]
\newtheorem{lemma}[theorem]{Lemma}
\newtheorem{proposition}[theorem]{Proposition}

\newtheorem{definition}[theorem]{Definition}

\newcommand{\F}{\mathbf{F}}

\newcommand{\N}{\mathbf{N}}

\newcommand{\Q}{\mathbf{Q}}

\DeclareMathOperator{\Sym}{Sym}
\DeclareMathOperator{\Alt}{Alt}
\DeclareMathOperator{\Mat}{Mat}
\DeclareMathOperator{\GL}{GL}

\DeclareMathOperator{\Cent}{Cent}

\DeclareMathOperator{\im}{im}

\title[On types of matrices]{On types of matrices and centralizers of matrices and permutations}%\author{John R. Britnell}
\author{John R. Britnell}
\address{Department of Mathematics, Imperial College London, London, SW7 2AZ}
\email{j.britnell@imperial.ac.uk}

\author{Mark Wildon}
\address{Department of Mathematics, Royal Holloway, University of London, Egham, Surrey TW20 0EX, United Kingdom}
\email{mark.wildon@rhul.ac.uk}


\begin{document}
\maketitle
\thispagestyle{empty}

\begin{abstract}
It is known that that the centralizer of a matrix
over a finite field depends, up to conjugacy, only on the type of the
matrix, in the sense defined by J.~A.~Green.
In this paper an analogue of the type invariant is defined that
in general captures more information; using this invariant the
result on centralizers is extended to arbitrary fields.
The converse is also proved: thus two matrices have conjugate centralizers
if and only if they have the same generalized type.
The paper ends with the analogous results for symmetric and alternating groups.
\end{abstract}

\section{Introduction}
The notion of the type of a matrix over a finite field
was defined by Green in his influential paper \cite{Green} on characters
of finite general linear group, generalizing early work of Steinberg \cite{Steinberg}.
In Green's definition, the type
of a matrix is obtained from its cycle type by formally replacing
each irreducible polynomial with its degree.
In \cite[Lemma 2.1]{Green} Green showed that two matrices
with the same type have isomorphic centralizer algebras. In \cite[Theorem 2.7]{BW_GLClasses}
the authors strengthened this result by proving that the centralizers are in fact
conjugate. In this paper we generalize Green's definition of type to matrices
over an arbitrary field, and prove the following theorem characterizing
all matrices with conjugate centralizers.

\begin{theorem}\label{thm:iff}
Let $K$ be a field and let $X$, $Y \in \Mat_n(K)$. The
centralizers of~$X$ and $Y$ in $\Mat_n(K)$ are conjugate 
by an element of
$\GL_n(K)$ if and only
if $X$ and~$Y$ have the same generalized type.
\end{theorem}

The definition of generalized type given in Section~\ref{s:def} below agrees with Green's for
fields with the unique extension property; these include finite fields, and also
algebraically closed fields. Thus an immediate corollary of Theorem~\ref{thm:iff}
is that two matrices over a finite field have the same type if and only if
their centralizers are conjugate. This gives the
converse of Theorem~2.7 of \cite{BW_GLClasses}. % holds.

The proof of Theorem~\ref{thm:iff} is given in Sections~\ref{s:if} and~\ref{s:onlyif} below.
In Section~\ref{s:if}
we prove that two matrices with the same generalized type have conjugate centralizers.
We obtain this result as a corollary of Theorem~\ref{thm:polytypes},
which states that two matrices have the same generalized type if and only if their
similarity classes contain representatives that are polynomial in one another.

In Section~\ref{s:onlyif} we prove the converse implication of Theorem~\ref{thm:iff},
that if two matrices have conjugate centralizers
then their generalized types agree.  This requires a number of `recognition' results
on centralizers that build on the work in~\cite{BW_GLClasses}. Some preliminary
results needed in both parts of the proof are collected in Section~\ref{s:prelim}.

An aspect of our work to which we would like to direct attention is our method,
in the proof of Theorem \ref{thm:polytypes}, for dealing with a possibly
inseparable field extension. This result is a generalization of
\cite[Theorem 2.6]{BW_GLClasses}, but the proof of the earlier result depends on
the existence of a Jordan--Chevalley decomposition, which can fail when the field is arbitrary.
We avoid this problem by means of Lemma~\ref{lemma:JC}, which offers a dichotomy:
if the  minimal polynomial of a matrix $X$ is a power of irreducible polynomial,
then  either~$X$ has a Jordan--Chevalley decomposition, or else~$X$ possesses a very
strong stability property under polynomial functions.

It is possible to make a similar statement about centralizers in symmetric groups, to the
effect that permutations with conjugate
centralizers have the same cycle type, except for certain `edge cases'. It is clear
that this result is directly
analogous to Theorem \ref{thm:iff}, and since we have not found it in the literature,
we have included it here. Section
\ref{s:Sn} contains this result (Theorem~\ref{thm:Sn}), and also the corresponding
result for centralizers
in alternating groups.

It is natural to ask whether the generalized type of a matrix is determined by the unit group of its centralizer.
In the case of a matrix $X$ over any field other than $\F_2$, the answer is that its type is indeed so determined; this follows
from Theorem~\ref{thm:iff} via the observation that any element of the centralizer algebra of $X$ is a sum of two units. For let
$Y\in \Cent(X)$, and consider the primary decomposition of $Y$; define $T$ to act as the identity on all but the unipotent summand of $Y$, and as
any non-identity, non-zero scalar on that summand; then $T$ and $Y-T$ are both units. Centralizers over the field~$\F_2$ are not always generated linearly
by their unit groups however, and for instance the centralizers of the two matrices
\[ 
\left( \begin{matrix} 1 & 0 \\ 0 & 0 \end{matrix} \right), \quad 
\left( \begin{matrix} 1 & 1 \\ 0 & 0
\end{matrix} \right)
\]
%\[
%\left(\begin{array}{cc}1&0\\0&0\end{array}\right), \quad \left(\begin{array}{cc}1&1\\0&0\end{array}%\right),
%\]
are distinct, although each has a trivial unit group.

\section{Types and generalized types}\label{s:def}

%We begin by recalling some necessary background.
Let $K$ be a field, let $n \in \N$, let $V$ be the $K$-vector space $K^n$,
and let $X\in\Mat_n(K)$;
we suppose throughout that matrices act on the right.
Let $V=\bigoplus U_i$ be a decomposition
of $V$ as a sum of $X$-invariant subspaces, on each of which the action of $X$ is indecomposable. Let $X_i$ be $X$
restricted to $U_i$. Then each $X_i$ is a cyclic matrix and
the minimum polynomial of $X_i$ is $f^t$, for some polynomial $f$ irreducible over $K$, and
some positive integer $t$. For each such
irreducible $f$, let $\lambda_f$ be the partition obtained by collecting together
the values of $t$ arising in this way (counted with multiplicity). Although the decomposition of $V$ is not in general
unique, the partitions $\lambda_f$ are invariants of $X$ and
collectively they determine $X$ up to similarity of matrices.

Suppose that $X$ is a matrix whose characteristic polynomial has the
irreducible factors $f_1,\dots, f_t$,
with respective degrees $d_1,\dots,d_t$,
and that the partition invariants corresponding to these polynomials are
$\lambda_1,\dots \lambda_t$ respectively.
The \emph{cycle type} of $X$ is
the formal product
$f_1^{\lambda_1}\cdots f_t^{\lambda_t}$.
We say that a matrix over a field $K$ is \emph{primary} if it has cycle
type~$f^\lambda$ for some
irreducible polynomial $f$ and partition~$\lambda$.
The \emph{type} of $X$,
as defined by Green in \cite[page~407]{Green}
is the formal product $d_1^{\lambda_1}\cdots d_t^{\lambda_t}$.

Green's definition of type makes sense when $K$ is an arbitrary field.
However Theorem 2.8 of \cite{BW_GLClasses}, which states that matrices over a finite
field with the same type
have conjugate centralizers, would not extend to matrices over arbitrary fields if this definition were in force.
To give an instance, let $X$ and $Y$ be the rational companion matrices of
%with minimal polynomials
the irreducible polynomials $f(x)=x^2-2$ and $g(x)=x^2-3$. %over $\Q$.
These matrices both have type $2^{(1)}$. Since
$X$ and $Y$ are cyclic we have that $\Cent X = \Q\langle X\rangle$ and
$\Cent Y = \Q\langle Y \rangle$. But~$X$ is not conjugate
to a polynomial in~$Y$, since the eigenvalues of $X$ and~$Y$
lie in distinct quadratic extensions of $\Q$.

This example, however,
suggests a very natural way of extending Green's definition which, as we shall
show, allows the
theorem we have mentioned to be generalized to infinite fields.

\begin{definition}\label{defn:eqv}
Let $K$ be a field, and let $\Phi$ be the set of irreducible polynomials over $K$.
Let $f$, $g \in \Phi$ and let $L$ be a splitting field for $fg$. We say
that $f$ is \emph{equivalent} to $g$ if whenever $\alpha \in L$ is a root
of $f$ there exists a root $\beta \in L$ of $g$ such that $K(\alpha) = K(\beta)$,
and vice versa. We denote equivalence by $f \sim g$, and
denote the equivalence class of $f$ by $[f]$.
\end{definition}

Since all splitting fields for $fg$ are isomorphic as extensions
of~$K$, this definition does not depend on the choice of $L$.

\begin{definition}\label{defn:gentype}
Let $X \in \Mat_d(K)$ %have partition invariants $(\lambda_f)$ for $f \in \Phi_$.
and let $\Phi_X$ be the set of irreducible polynomials
for which the partition invariant $\lambda_f$ of $X$ is non-empty.
We define the \emph{generalized type} of $X$ to be the formal product
\[\prod_{f \in \Phi_X} % \prod_{f \in \Phi_f}
 [f]^{\lambda_f} \]
in which the order of terms is unimportant.
\end{definition}

We note that if
$K$ has the unique extension property (and in particular, if $K$ is finite),
then two polynomials are
equivalent under $\sim$ if and only if they have the same degree.
Our definition of generalized type therefore agrees with Green's in this case.

\section{Preliminary results}\label{s:prelim}

We require two general results from \cite{BW_GLClasses}.
For $d\in \N$, and for a partition~$\lambda$, we write
$d\lambda$ for the partition with $d$ parts of size $i$ for every
part of size~$i$ in~$\lambda$.
For a partition $\lambda$ we write $N(\lambda)$ for the
similarity class of nilpotent matrices of type $1^\lambda$.
The dominance order on partitions will be denoted by~$\unlhd$.

%We write $\unlhd$
%for the dominance order on partitions.

\begin{proposition}[{\cite[Proposition 2.2]{BW_GLClasses}}]\label{prop:nilpclass}
Let $M$ be a matrix of primary type~$d^\lambda$. If the cycle
type of $M$ is $f^\lambda$ then $f(M)$ is nilpotent and $f(M) \in N(d\lambda)$.
\end{proposition}

%Moved from \S 4

\begin{proposition}[{\cite[Proposition 2.4]{BW_GLClasses}}]\label{prop:dominance}
Let $X$ be a primary matrix of type~$d^{\lambda}$ with entries from a
field~$K$,
and let $h \in K[x]$ be a polynomial.
The type of $h(X)$ is~$e^\mu$ for some $e$ dividing $d$,
and some partition $\mu$ such that $e|\mu|=d|\lambda|$
and~\hbox{$e \mu \unlhd d \lambda$}.
\end{proposition}

We also need the following result giving the dimension of the centralizer of a matrix.
If $\lambda$ is a partition with exactly $m_i$ parts
of size $i$, we define
%\[ F(\lambda) = \sum_{j < k} 2jm_jm_k + \sum_{j} jm_j^2. \]
\[ F(\lambda) = \sum_{j} \sum_k \min(j,k)m_jm_k. \]

\begin{proposition}\label{prop:centdim}
Let $K$ be a field and let $X \in \Mat_n(K)$ have type $d_1^{\lambda_1} \ldots d_t^{\lambda_t}$.
Then $\dim_K \Cent X = \sum_{i=1}^t d_i F(\lambda_i)$.
\end{proposition}

\begin{proof}
Let $V = K^n$. Since the subspaces corresponding to the primary decomposition of $X$
are preserved by $\Cent X$, we may reduce to the case where $X$ is a primary matrix
of cycle type $f^\lambda$. Let the degree of  $f$ be $d$.

Given a vector $v \in V$ we say that $v$ has \emph{height} $h \in \N$
if \hbox{$f(X)^{h-1}v \not=0$} and $f(X)^h v = 0$.
Let $V = \bigoplus_{i=1}^r U_i$ be a direct sum decomposition of~$V$ into indecomposable
$X$-invariant subspaces such that the dimension of $U_i$ is equal to the $i$th part of
$\lambda$. Let $u_i$ be a cyclic vector
generating $U_i$. If $h$ is a part of $\lambda$ then the images
of the $m_h$ cyclic vectors of height $h$ can be chosen freely from the subspace of $V$
of vectors of height at most $h$. This subspace has dimension
\[ d \bigl( h \sum_{j \ge h} m_j + \sum_{k < j} k m_k \bigr). \]
The proposition now follows by a straightforward counting argument.
\end{proof}

As a corollary, we see that the dimension of the centralizer of a matrix depends on
the field of definition only through the information captured by its type.
%Signpost to the technical proposition about field extensions?

In the special case of nilpotent matrices this proposition is well known.
For two equivalent formulations
see Propositions 3.1.3 and 3.2.2 in \cite{OMeara}. The first implies that
$F(\lambda) = \sum (2i-1)\ell_i$, where $\ell_i$ is the $i$th part of $\lambda$;
the second, which is originally due to Frobenius, gives
$F(\lambda) = \sum {\ell_i'}^2$, where $\ell'_i$ is the $i$th part of the conjugate
partition to $\lambda$.

\section{Matrices with conjugate centralizers}\label{s:if}

The aim of the remainder of this section is to
prove Theorem~\ref{thm:polytypes} and hence the `if' direction of Theorem~\ref{thm:iff}.

\begin{proposition}\label{prop:polymult}
Let $X$ be nilpotent of class $f^\lambda$, where $f$ has degree $d$
and $\lambda$ is a partition with at least one part
of size greater than $1$. Let $r(x)$ be a polynomial. Then $r(X)\in N(d\lambda)$ if and only if $r(x)$ is divisible by
$f(x)$ but not by $f(x)^2$.
\end{proposition}

\begin{proof}
It is clear that $r(X)$ is nilpotent if and only if $f(x)$ divides $r(x)$.
Let $r(x)=g(x)f(x)^a$ where
$g(x)$ is coprime to $f(x)$.
Since $g(X)$ is invertible, and commutes with $f(X)^a$, we see that the dimensions
of the kernels of $r(X)^i$ and $f(X)^{ai}$ are the same for all $i$. Since these dimensions determine the similarity class of $X$, it follows
that $r(X)$ is similar to $f(X)^a$. By Proposition~\ref{prop:nilpclass}
we have $f(X) \in N(d \lambda)$. Hence if $a = 1$ then $r(X) \in N(d\lambda)$,
while if $a > 1$ then $r(x) \not\in N(d\lambda)$, since $\lambda$ has a part of size 
greater than~$1$.
%Now it is clear that $f(X)^a$ is of different class from $f(X)$ if $a>1$
%(since $\lambda$ has a part greater than $1$), and the proposition follows.
\end{proof}

Let $X$ be a matrix over a field $K$. Recall that an additive Jordan--Chevalley decomposition of $X$ is a
decomposition $X=S+N$, where $S$ and $N$ are matrices over $K$ such that $S$ is semisimple, $N$ is nilpotent, and
$SN=NS$. If a Jordan--Chevalley decomposition of $X$ exists then it is unique, and both $S$ and $N$ are polynomial
in $X$. Over a perfect field, every matrix admits a Jordan--Chevalley decomposition, and the proof of
\cite[Theorem 2.6]{BW_GLClasses} (in which the field is finite) relies on this fact. Over an arbitrary field these
decompositions do not generally exist; but the following lemma allows us to compensate for their lack.

\begin{lemma}\label{lemma:JC}
Let $X$ be a primary matrix over a field $K$ of cycle type
$f^\lambda$. Let $r$ be a polynomial over $K$ such that $r(X)$
has class $f^{\mu}$. If $\mu\neq \lambda$, then~$X$ has a Jordan--Chevalley decomposition over $K$.
\end{lemma}
\begin{proof}
If all parts of $\lambda$ are equal to $1$, then $X$ is semisimple, and has
an obvious %a trivial
Jordan--Chevalley decomposition.
So we suppose that $\lambda$ has a part greater than $1$.

Let $d$ be the degree of $f$. Since $r(X)$ has class $f^\mu$, we see from Proposition \ref{prop:nilpclass} that
$(f\circ r) (X)$ is nilpotent and lies in the similarity class $N(d\mu)$.
It follows that $f$ divides $f\circ r$. Let $f\circ r = gf$ for
some polynomial $g$. If $g$ is coprime with $f$, then by
Proposition~\ref{prop:polymult} we see that $gf(X)$ is the
same nilpotent class as $f(X)$, and so we have $\mu=\lambda$.

Suppose, then, that $g$ is divisible by $f$, and so $f\circ r = hf^2$ for some polynomial~$h$.
Observe that
\[
f\circ(r\circ r) = (f\circ r)\circ r = hf^2 \circ r = (h\circ r)(f\circ r)^2 = (h\circ r) h^2f^4.
\]
Similarly, writing $r^{(a)}$ for the $a$-th power of $r$ under composition, we see that $f\circ r^{(a)}$ is divisible
by $f^{2^a}$. So for sufficiently large $a$, we have $(f\circ r^{(a)})(X)=0$.

Let $L$ be a splitting field for $f$ over $K$. Notice that the polynomial $r$ acts on the roots of $f$ in
$L$ by permuting them,
since these roots are the eigenvalues of both $X$ and~$r(X)$.
We may suppose (by increasing
$a$ as necessary) that $r^{(a)}$ fixes each root of $f$. Then certainly $r^{(a)}(X)\neq 0$, and since $f\bigl(r^{(a)}(X)\bigr) = 0$, it follows that
$S=r^{(a)}(X)$ is a semisimple matrix with minimum polynomial $f$. But since any eigenvector of $X$ over $L$
is an eigenvector of $S$ with the same eigenvalue, we see that $N=X-S$
%(which is a primary matrix over $K$)
must be
nilpotent. So we have found a Jordan--Chevalley decomposition $S+N$ for~$X$.
\end{proof}

\begin{theorem}\label{thm:polytypes}
Let $K$ be a field, and let $X$, $Y\in\Mat_d(K)$. Then $X$ and $Y$ have the same
generalized type if and only if there exist polynomials
$p$ and~$q$ such that $p(X)$ is similar to $Y$ and $q(Y)$ is similar to $X$.
\end{theorem}
\begin{proof}
This is the generalization to an arbitrary field of \cite[Theorem 2.6]{BW_GLClasses}, and only part of the proof is
complicated by the necessity of appealing to Lemma~\ref{lemma:JC}. We shall therefore present the unaffected parts of
the argument very concisely, referring the reader to our earlier paper for a
%rather
gentler exposition.

We  show first that if $X$ and $Y$ have the same generalized type then
there exists a polynomial $p$ such that $p(X)$ is similar to
$Y$. By an appeal to the Chinese Remainder Theorem, we see that it is enough to prove the result in the case that~$X$ a
primary matrix of cycle type $f^\lambda$ for some irreducible polynomial $f$
and some partition $\lambda$. By hypothesis there exists an irreducible
polynomial $g$ with $f \sim g$, such that $Y$ has cycle type $g^\lambda$.

Let $\alpha$  be a root of $f$ in an extension field of $K$ in which
$f$ and $g$ split. Since \hbox{$f \sim g$} there exists a root $\beta$ of $g$
and  polynomials $r$ and $s$ over $K$
such that $r(\alpha)=\beta$ and $s(\beta)=\alpha$.
Now if $\alpha'$ is any root of $f$ then, since $\alpha$ is sent to $\alpha'$
by an automorphism of $L$ fixing $K$, we see that $r(\alpha')$ is a root of $g$
and $s(r(\alpha')) = \alpha'$. It follows that
$r(X)$ has class $g^\mu$ for some partition $\mu$ and $(s\circ r)(X)$
has class $f^\nu$ for some partition $\nu$.
Since $(s \circ r)(X)$
is polynomial in $r(X)$, it follows from Proposition~\ref{prop:dominance}
that $\lambda \unrhd \mu \unrhd \nu$.

Suppose that $\lambda = \nu$. Then
the classes $f^\lambda$ and $g^\lambda$ are polynomial in one another,
witnessed by the polynomials $r$ and $s$.

Suppose, on the other hand, that $\nu\neq\lambda$.
Then by Lemma \ref{lemma:JC}, the matrix $X$ has a Jordan--Chevalley
decomposition $X=S+N$. It is now easy to see that
$r(S)+N$ is the Jordan--Chevalley decomposition for some matrix
$Y'$ belonging to the class $g^\pi$ for some partition $\pi$. Since both $S$ and $N$ are polynomials in $X$, we have that %it follows that
$r(S)+N$ is a polynomial in $X$. Similarly, we see that
$(s\circ r)(S)+N$ is polynomial in $Y'$. Since
$s\circ r$ fixes the eigenvalues of $X$
we must have $(s\circ r)(S)=S$, and so $X$ is polynomial
in $Y'$. But now it follows from
Proposition \ref{prop:dominance} that $\lambda=\pi$, and so the classes $f^\lambda$
and $g^\lambda$ are polynomial in one another in this case too.

Conversely, suppose that $p(X)$ is similar to $Y$ and $q(Y)$ similar to $X$. Since the number of %(non-trivial)
summands
in the primary decomposition of $p(X)$ is at most the number in that of $X$, we see that the primary decomposition of
$X$ and $Y$ have the same number of summands. Let $X_f$ be the summand of $X$ corresponding to the polynomial $f$, and
let $Y_g$ be the summand of $Y$ similar to $p(X_f)$, corresponding to the polynomial $g$. Since $p$ sends the eigenvalues of $X$ (in a suitable extension field)
to eigenvalues of $Y$, it is clear that $K(\alpha)$ embeds into $K(\beta)$.
By symmetry we have $K(\alpha) = K(\beta)$ and so $f \sim g$.
%But we may assume that $q(Y_g)$ is similar to $X_f$, and it follows that
%$f\sim g$.
Now it follows from Proposition \ref{prop:dominance} that the partition invariants $\lambda_f$ of $X$ and
$\lambda_g$ of $Y$ are the same. So $X$ and $Y$ have the same type.
\end{proof}

%By a similar argument to that used to prove Theorem 2.7
%in \cite{BW_GLClasses}, 
We now
obtain one half of Theorem~\ref{thm:iff}.

\begin{proof}[Proof of `if' direction of Theorem~\ref{thm:iff}]
By Theorem \ref{thm:polytypes} there exist polynomials 
$p$ and $q$ such that $p(X)$ is similar to~$Y$ and~$q(Y)$ is
similar to $X$. Now $\Cent X$ is a subalgebra of $\Cent p(X)$ and so $\Cent X$ is
conjugate to a subalgebra of $\Cent Y$. Similarly $\Cent Y$ is a subalgebra of $\Cent q(Y)$,
and so $\Cent Y$ is conjugate to a subalgebra of $\Cent X$.
It follows from considering the dimensions of these subalgebras that
$\Cent X=\Cent p(X)$ and that $\Cent Y=\Cent q(Y)$.
\end{proof}

\section{Recognizing the generalized type of a matrix from its centralizer}\label{s:onlyif}

Throughout this section we let $K$ be a field.
Let $X$ and $Y$ be matrices  in $\Mat_n(K)$ with
conjugate centralizer algebras. By replacing
$Y$ with an appropriate conjugate,
we may assume that in fact $\Cent X$ and $\Cent Y$ are equal.
We shall show that $X$ and $Y$ have the same generalized type.

The proof proceeds by a series of reductions. We first prove the result
for nilpotent matrices, then for primary matrices, and finally, for
general matrices.

\begin{lemma}\label{nilpotent}
If $M$ and $N$ are nilpotent matrices, and $\Cent M = \Cent N$,
then $M$ and $N$ are conjugate by an element of
$\GL_n(K)$. \end{lemma}
\begin{proof}
We use results from Section 3 of \cite{BW_GLClasses}. Let $A=\Cent M$.
Let the partition associated with $M$ have
$m_h$ parts of size $h$ for each $h\in\N$. By Propositions 3.4 and 3.5 of \cite{BW_GLClasses}, for each $h$ such that
$m_h>0$, the $A$-module $V$ has a composition factor of dimension $m_h$ which appears with multiplicity $h$; these are
all of the composition factors of $V$. Thus the
similarity class of $M$ can be recovered from a composition series for $V$.
\end{proof}

\begin{lemma}\label{doublePD} Suppose that $X$, $Y\in\Mat_n(K)$ have equal
 centralizers. Then the primary decompositions
of $V$ as a $K\langle X\rangle$-module and as an $K\langle Y\rangle$-module
have the same subspaces of $V$ of summands.
%are the same (i.e., they have the same subspaces
%as summands).
\end{lemma}
\begin{proof} Since $X$ and $Y$ commute
%It is obvious that $X$ and $Y$ commute; therefore
we may form the simultaneous primary decomposition
\[\label{eq1}\tag{$\star$}
V=\bigoplus_{f,g}V_{f,g},
\]
where the direct sum is over pairs of irreducible polynomials in $K[x]$
and $V_{f,g}$ is the maximal subspace of $V$ on which
both $f(X)$ and $g(Y)$ have nilpotent restrictions.
Suppose that $V_{f,g_1}$ and $V_{f,g_2}$ are both non-trivial, where $g_1$ and $g_2$ are
distinct irreducible polynomials.
Let $v$ generate $V_{f,g_1}$ as a $K\langle f(X) \rangle$-module
and let $w$ be a vector in the kernel of the restriction of $f(X)$
to $V_{f,g_2}$. There is a $K\langle X \rangle$-endomorphism of $V$
that maps $v$ to $w$. Such an endomorphism corresponds to  matrix $Z \in \Cent X$
such that $V_{f,g_1}Z$ intersects non-trivially with $V_{f,g_2}$.
On the other hand, no such $Z$ can belong to $\Cent Y$;
this contradicts the assumption that
$\Cent X=\Cent Y$.

It follows that the decomposition ($\star$) is simply the primary decomposition of $V$ as a $K\langle X\rangle$-module.
The lemma follows by symmetry.
\end{proof}

To complete the proof in the primary case we need the following lemma
and proposition
describing
how the type and centralizer algebra
of a matrix change on field extensions.
%In particular, we prove that the dimension of a centralizer
%of a matrix is independent of its type.
Given a partition $\lambda$, let $\lambda \times p$
denote the partition obtained by multiplying all of the parts of $\lambda$ by $p$.

\begin{lemma}\label{lemma:insepcyclic}
Suppose that $K$ has prime characteristic $p$.
Let $X \in \Mat_n(K)$ be a primary matrix of cycle type $f^\lambda$ where
$f(x^p) \in K[x]$ is an inseparable irreducible polynomial. Let $L$ be an extension field
of $K$ containing the $p$th roots of the coefficients of $f$, and let $g \in L[x]$
be such that $g(x)^p = f(x^p)$. Then the cycle type of $X$ over $L$ is~$g^{\lambda \times p}$.
\end{lemma}

\begin{proof}
It is sufficient to prove the lemma when $X$ is cyclic and so $\lambda$ has a single
part. Suppose that $\lambda = (h)$. Let $V = K^n$ regarded as a $K\langle X \rangle$-module.
Since $V \cong K[x]/(f(x^p)^h)$, there is an isomorphismism of $L\langle X \rangle$-modules
%\[ V \otimes_K L \cong K[x]/(f(x^p)^h) \otimes_K L \cong L[x]/(f(x^p)^h) = L[x]/g(x)^{ph}. \]
%\[ V \otimes_K L \cong \frac{K[x]}{(f(x^p)^h)}
%  \otimes_K L \cong \frac{L[x]}{(f(x^p)^h)} = \frac{L[x]}{g(x)^{ph}}. \]
\[ V \otimes_K L \cong \frac{K[x]}{\langle f(x^p)^h \rangle}
  \otimes_K L \cong \frac{L[x]}{\langle f(x^p)^h \rangle}
  = \frac{L[x]}{\langle g(x)^{hp} \rangle}. \]
Hence $X \otimes 1$ acts as a cyclic matrix on $V \otimes_K L$ with minimal polynomial
$g(x)^{hp}$. Therefore $X \otimes 1$ has cycle type $g^{(hp)}$, as required.
\end{proof}
\begin{proposition}\label{prop:centext}
Let $X \in \Mat_n(K)$ be a primary matrix of cycle type $f^\lambda$
and let $L$ be a splitting field for $f$.
%Write $X_L$ for the matrix $X$,
%regarded as an element of $\Mat_n(L)$.
Under the isomorphism between $\Mat_n(L)$ and $\Mat_n(K) \otimes L$,
the image of $\Cent_{\Mat_n(L)}X$ is \hbox{$\Cent_{\Mat_n(K)} X \otimes 1$}.
Moreover if $f$ has distinct roots $\alpha_1, \ldots, \alpha_d$ in $L$,
where each root of $f$ has multiplicity $p^a$,
then the cycle type of $X$,
regarded as an element of $\Mat_n(L)$, is
\[ (x-\alpha_1)^\lambda \ldots (x-\alpha_d)^\lambda \]
if $f$ is separable, and
\[ (x-\alpha_1)^{\lambda
\times p^a} \ldots (x-\alpha_d)^{\lambda \times p^a} \]
if $f$ is inseparable and each root of $f$ in $L$ has multiplicity $p^a$.
\end{proposition}

\begin{proof}
Clearly $\Cent_{\Mat_n(K)} X \otimes L$ is isomorphic to a
subalgebra of $\Cent_{\Mat_n(L)} X$.
We shall prove that the dimensions are the same, and at the same time
establish the other claims in the proposition.

Suppose first of all that $f$ is separable. Then $f$ factors
as $(x-\alpha_1)\ldots(x-\alpha_d)$ in $L[x]$.
Since the $\alpha_i$ are conjugate by automorphisms of $L$ fixing $K$,
there is a partition $\mu$
such that, over $L$, the cycle type of $X$ is $(x-\alpha_1)^\mu \ldots (x-\alpha_d)^\mu$.
Therefore $f(X)$, regarded as a matrix over $L$, lies in the similarity class $N(d\mu)$.
But by Proposition~\ref{prop:nilpclass}, we have $f(X) \in N(d\lambda)$, and so
$\lambda = \mu$. Proposition~\ref{prop:centdim} now implies that
\[ \dim_L \Cent_{\Mat_n(L)} X = d F(\lambda) = \dim_K \Cent_{\Mat_n(K)} X. \]

Now suppose that $f$ is inseparable. Let $K$ have prime characteristic $p$
and suppose that $f$ factors as $(x-\alpha_1)^{p^a}\ldots(x-\alpha_d)^{p^a}$
where $a \ge 1$ and the $\alpha_i$ are distinct. Let $g(x) = (x-\alpha_1)\ldots(x-\alpha_d)$.
Lemma~\ref{lemma:insepcyclic} implies that
the cycle type of $X$ over the field extension of $K$ generated by the coefficients of $g$
is $g^{\lambda \times p^a}$. Since $g$ is separable, it now follows that the cycle
type of $X$ over $L$ is $(x-\alpha_1)^{\lambda \times p^a} \ldots (x-\alpha_d)^{\lambda \times
p^a}$. Proposition~\ref{prop:centdim} implies that
\[ \dim_L \Cent_{\Mat_n(L)} X = d F(\lambda \times p^a) = dp^a F(\lambda) =
\dim_K \Cent_{\Mat_n(K)} X, \]
again as required.
\end{proof}

\begin{proposition}\label{prop:primary} Let $f$ and $g$ be irreducible polynomials over $K$. %with degrees $c$ and $d$ respectively.
Let $X$, $Y\in\Mat_n(K)$ have cycle types $f^{\lambda}$ and $g^{\mu}$ respectively, and suppose that $\Cent X =\Cent Y$.
%Then the degrees of $f$ and $g$ are equal and $\lambda=\mu$.
Then $f \sim g$ and $\lambda=\mu$.

\end{proposition}
\begin{proof} We shall work over a splitting field $L$ for the product
$fg$. By the first part of
Proposition~\ref{prop:centext} the centralizers of $X$ and $Y$
in $\Mat_n(L)$ are equal.

Let $f$ have distinct roots $\alpha_1,\ldots,\alpha_c$
and let $g$ has distinct roots $\beta_1,\ldots,\beta_d$ in $L$.
By Proposition~\ref{prop:centext}
if $K$ has characteristic zero then the cycle types of $X$ and $Y$
over $L$ are respectively
\begin{align*}
& (x-\alpha_1)^\lambda\cdots (x-\alpha_c)^\lambda,\\
& (x-\beta_1)^{\mu}\cdots (x-\beta_d)^{\mu},
\end{align*}
while if $K$ has prime characteristic $p$ then there exists $a,b\in \N_0$
such that the cycle types are respectively
\begin{align*}
& (x-\alpha_1)^{\lambda \times p^a}\cdots (x-\alpha_c)^{\lambda \times p^a},\\
& (x-\beta_1)^{\mu \times p^b}\cdots (x-\beta_d)^{\mu \times p^b}.
\end{align*}
Since $X$ and $Y$ have the same centralizer over $L$, it follows from Lemma \ref{doublePD} that their primary decompositions have the same number of summands, and so we have $c=d$ in
both cases.
Furthermore, the primary decompositions of $X$ and $Y$ over~$L$
have the same subspaces as summands. Let this decomposition be
$\bigoplus V_i$
where  $X$ has the eigenvalue $\alpha_i$ and $Y$ the eigenvalue $\beta_i$ on
$V_i$. Let $X_i$ and
$Y_i$ denote the restrictions of $X$ and $Y$ to $V_i$, respectively.
Then it is clear that
\[
\Cent X= \bigoplus_i\Cent X_i,\quad \Cent Y=\bigoplus_i\Cent Y_i,
\]
and since $\Cent X=\Cent Y$ it follows that $\Cent X_i=\Cent Y_i$
for all $i$. But $\Cent X_i=\Cent(X_i-\alpha_iI)$
and $\Cent Y_i=\Cent(Y_i-\beta_iI)$, and so the nilpotent matrices $X-\alpha_iI$ and $Y-\beta_iI$ have the same
centralizer; by Lemma \ref{nilpotent} they must be conjugate.
In the separable case $X-\alpha_iI$ has
the partition $\lambda$ and
$Y-\beta_iI$ has the partition $\mu$, and so we have $\lambda=\mu$, as required.
In the inseparable case $X-\alpha_i I$ has the partition $\lambda \times p^a$
and $Y-\beta_iI$ has the partition $\mu \times p^b$. Since
$c=d$ the partitions $\lambda$ and $\mu$ are partitions of the same number.
Hence we have $a=b$ and so $\lambda = \mu$, as required.

It remains to show that $f \sim g$. For this we shall work over the original
field~$K$.
Take $r \in \N$ such that 
$f(X)^{r-1} \not= 0$ and $f(X)^{r} = 0$. The action of $X$ on  $\im f(X)^{r-1}$
is semisimple, since it acts as a direct sum of copies
of the irreducible companion matrix $C$ of $f$.
The $X$-endomorphisms of this
subspace form a full matrix algebra with coefficients in
$K\langle C \rangle$. % \cong K[x]/f(x)$.
The centre of this algebra consists of
the diagonal matrices with coefficients in $K\langle C\rangle$.
Therefore $\Cent_{\Mat_n(K)} X$ determines $K\langle C \rangle$.
Hence we have $K\langle C\rangle = K\langle D\rangle$ where
$D$ is the companion matrix for $g$. It follows that if
$\alpha$ is an eigenvalue of
$C$ then there
is a polynomial $s\in K[x]$ such that $s(D)$ has $\alpha$
as an eigenvalue.
But the eigenvalues of $S(D)$ are $\{ s(\beta_1), ..., s(\beta_d)\}$
so $K(\alpha) = K(\beta_j)$ for some~$j$. Therefore $f \sim g$.
\end{proof}

We are now ready to prove the other half of Theorem~\ref{thm:iff}

\begin{proof}[Proof of `only if' direction of Theorem~\ref{thm:iff}]
By Lemma \ref{doublePD} the primary decompositions of $X$ and $Y$ are the same.
Let
\[
V=\bigoplus_{i=1}^t V_i,
\]
where for each $i$ there exist irreducible polynomials $f_i$ and $g_i$ such that $f_1,\dots, f_t$ are distinct,
$g_1,\dots,g_t$ are distinct, and both $f_i(X)$ and $g_i(Y)$ are nilpotent on their restriction to $V_i$. Now by Lemma~\ref{doublePD}, it follows that $\Cent X_i=
\Cent Y_i$, where $X_i$ and $Y_i$ are the restrictions
of $X$ and $Y$ to $V_i$. But then it follows from Proposition~\ref{prop:primary}
that $f_i \sim g_i$ and that the partitions
associated with these polynomials are equal. Therefore the generalized types of
$X$ and~$Y$ are the same.
\end{proof}

\section{Centralizers in symmetric and alternating groups}\label{s:Sn}

Theorem \ref{thm:iff} is analogous to a result for symmetric groups, which, since we have been unable to find it in
the literature, we record here.
Let $g$, $h$ be elements of the symmetric group $S_n$ of all
permutations of $\{1,\ldots, n\}$.
We write $g=v_1\cdots v_n$, where $v_i$ is the product of the cycles
of $g$ of length $i$. Similarly, we write $h=w_1\cdots w_n$.
%We require the following definitions.

\begin{definition}{\ }\label{defn:loceqv}
\vspace{-3pt}
\begin{enumerate}
\item If there exists $k_i$ such that $w_i=v_i^k$, then we say that $g$ and $h$ are \emph{locally equivalent} at $i$.
\item We say that $g$ and $h$ are \emph{equivalent} if they are locally equivalent at $i$ for all $i \in \{1,2,\ldots, n\}$.
\item If $S\subseteq \{1,\dots,n\}$ and if $g$ and $h$ are locally equivalent at all $i\notin S$, but not locally equivalent
at $i \in S$, then we say that there is a \emph{local variation} at $S$.
\end{enumerate}
\end{definition}

\begin{theorem}\label{thm:Sn}
Let $g$ and $h$ be elements of $S_n$ whose centralizers in
$S_n$ are equal. Either $g$ and $h$ are equivalent, or there is a local variation at $\{1,2\}$
described by one of the following statements:
\begin{enumerate}
\item $v_1v_2$ is conjugate to $(12)$ and $w_1w_2$ is simultaneously conjugate to $(1)(2)$, or vice versa.
\item $v_1v_2$ is conjugate to $(12)(3)(4)$ and $w_1w_2$ is simultaneously conjugate to $(1)(2)(34)$, or vice versa.
\end{enumerate}
\end{theorem}

\begin{proof} Let $X_i$ be the support of $v_i$. Then
\[
\Cent_{S_n}(g) \cong \bigoplus_i \Cent_{\Sym(X_i)}(v_i).
\]
If $g$ and $h$ are locally equivalent at $i$, then the support of $w_i$ is $X_i$,
and clearly
$\Cent_{\Sym(X_i)}(w_i)=\Cent_{\Sym(X_i)}(v_i)$. It follows easily that if $g$ and $h$ are equivalent,
then their centralizers are equal.

For the converse, let $G$ be the centralizer of $g$ in $S_n$. Let
$\alpha\in\{1\dots n\}$ be a point in $X_i$. Note that $G$ permutes
the orbits of $g$ of length $i$ transitively, as blocks for its action.
Thus  the orbit
$\alpha^G$ is equal to $X_i$.
Let $G_\alpha$ be the stabilizer of $\alpha$ in~$G$. It is not
hard to show that that $G_\alpha$ acts transitively on the points
in the cycles of length $i$ not containing $\alpha$, and fixes
the points lying in the same
$g$-cycle as $\alpha$. Thus the set $F_\alpha$ of fixed points of $G_\alpha$
consists precisely of the $i$ points lying in the same $g$-cycle as $\alpha$,
\emph{except} when $i=1$ and $g$ has exactly two fixed points.
Therefore, when we attempt
to reconstruct the orbits of $g$ from the permutation action of~$G$, the
ambiguities arise
precisely from the local variations in the statement of the theorem.
Furthermore since $(12)$ and $(1)(2)$
have the same centralizer in $S_2$, and since $(12)$ and $(34)$ have the same centralizer in
$S_4$, there is no possibility of resolving these ambiguities.

We shall assume that we are not in this exceptional case. Suppose that $g$ has $j$ cycles of length $i$.
We have seen that the set $X_i$ is determined by the permutation
action of $G$. We observe that $G$ contains an element
which acts as a full cycle $c$ on $X_i$. Let $g_i$ be the restriction
of $g$ to $X_i$. Since the centralizer of $c$ in $\Sym(X_i)$
is the cyclic group $\langle c\rangle$, we see that $g_i=c^m$ for some $m$. Since $c$ has order $ij$, it is clear that
$m=jk$ for some $k$ coprime with $i$.

Now if $h$ is another permutation whose centralizer in $S_n$ is $G$, and if $h_i$ is the restriction of $h$ to $X_i$,
then we must similarly have that $h_i=c^{j\ell}$, where $\ell$ is coprime with $i$. Now since $k$ and $\ell$ are
invertible modulo $i$, we have $g_i=h_i^{k/\ell}$ and $h_i=g_i^{\ell/k}$. So $g$ and $h$ are locally equivalent at
$i$ as required.
\end{proof}

An obvious consequence of Theorem~\ref{thm:Sn} is that if two elements $x$ and $y$ of~$S_n$ have
centralizers which are isomorphic as permutation groups, then either $x$ is conjugate  to $y$, or else there is a
unique transposition~$t$ such that $t$ centralizes $y$ and
$x$ is conjugate to $ty$. We remark that this conclusion does not
hold if the centralizers of $x$ and $y$ are isomorphic merely as abstract groups. As an example, suppose that
$n=2k\ell+k+\ell-1$, where $k$ and $\ell$ are greater than $1$, and such that
$k$, $\ell$, $2k-1$ and $2\ell-1$ are pairwise
coprime. Let $x$ and $y$ be permutations such that $x$ has cycles of lengths
$k$, $2\ell -1$ and $\ell(2k-1)$, and $y$ has cycles of lengths $\ell$, $2k-1$ and $k(2\ell-1)$. Then
$x$ and $y$ have no cycle lengths in common, but each has a centralizer that is cyclic of order $k\ell(2k-1)(2\ell-1)$.

Finally, it is  worthwhile to state  the analogous result to Theorem~\ref{thm:Sn} for the alternating
groups~$A_n$. We shall not prove it here;
the proof follows similar lines to that of Theorem~\ref{thm:Sn}, but is complicated slightly by the fact that centralizer
$G$ of an element $g$ in $A_n$ is
not in general a direct product of permutation groups on the sets $X_i$, though it
has index at most $2$
in such a product: in fact the restriction
of $G$ to the set $X_i$ acts either as $\Cent_{\Alt(X_i)}(v_i)$ or as $\Cent_{\Sym(X_i)}(v_i)$, depending on whether
the cycles of $g$ of length other than $i$ have distinct odd lengths.

\begin{theorem}
Let $g$ and $h$ be elements of $A_n$ whose centralizers in $A_n$ are equal. Then either $g$ and $h$ are equivalent,
or one of the following statements is true.
\begin{enumerate}
\item There is a local variation at $\{1,2\}$, with $v_1v_2$ conjugate to $(12)(3)(4)$
and $w_1w_2$ simultaneously conjugate to $(1)(2)(34)$.
\item There is a local variation at $\{2\}$, with $v_2$ being conjugate to $(12)(34)$
and $w_2$ simultaneously conjugate to $(13)(24)$. Elsewhere, each of $g$ and $h$ has only odd cycles of
distinct lengths.
\item There is a local variation at $\{1,3\}$, with $v_1v_3$ conjugate to $(123)$
and $w_1w_3$ simultaneously conjugate to $(1)(2)(3)$. Elsewhere, each of $g$ and $h$ has only odd cycles of
distinct lengths.
\item For some odd integer $m$ there is a local variation at $\{m\}$, with $v_m$ and $w_m$ each having exactly two
cycles. Each cycle of $w_m$ is a power of a cycle of~$v_m$, but the two exponents, taken modulo $i$, are distinct. Elsewhere, each of $g$ and~$h$ has only odd cycles of
distinct lengths.
\end{enumerate}
\end{theorem}


\begin{thebibliography}{99}

\bibitem{BW_GLClasses} John R. Britnell and Mark Wildon, `On types and classes of commuting matrices over finite fields',
\textit{J. London Math.\ Soc.\ }83 (2011) 470--492.
\bibitem{Green} J. A. Green, `The characters of the finite general linear groups', \textit{Trans.\ Amer.\ Math.\ Soc.\ }
\bibitem{OMeara}
Kevin C. O'Meara, John Clark, and Charles I. Vinsonhaler,
     `Advanced topics in linear algebra',
Oxford University Press, Oxford, 2011.

\bibitem{Steinberg} R. Steinberg, `A geometric approach to the representations of the full linear group over a Galois
field', \textit{Trans.\ Amer.\ Math.\ Soc.\ }71 (1951) 274--282.
\end{thebibliography}
\end{document}